\documentclass[11pt]{amsart}
\usepackage[letterpaper,left=75pt,right=75pt]{geometry}
\allowdisplaybreaks[0]
\usepackage{float, comment}
\usepackage{graphicx}
\usepackage{amssymb}
\usepackage{color}
\usepackage{hyperref}

\numberwithin{equation}{section}

\newcommand{\bH}{\mathbb{H}}

\newcommand{\bR}{\mathbb{R}}

\newcommand{\cC}{\mathcal{C}}
\newcommand{\cP}{\mathcal{P}}
\newcommand{\cG}{\mathcal{G}}

\newcommand{\cN}{\mathcal{N}}
\newcommand{\SL}{\mathrm{SL}}

\newcommand{\cS}{\mathcal{S}}

\newcommand{\vol}{\mathrm{vol}}

\newcommand{\fL}
{\mathfrak{L}}
\newcommand{\fS}
{\mathfrak{S}}

\newcommand{\K}{\mathrm{Kiss}}
\newcommand{\Hull}{\mathrm{Hull}}

\newtheorem{Theorem}{Theorem}

\newtheorem{Prop}[Theorem]{Proposition}
\newtheorem{lemma}[Theorem]{Lemma}
\newtheorem*{lemma*}{Lemma}
\newtheorem*{theorem*}{Theorem}

\theoremstyle{definition} 

\newtheorem{remark}[Theorem]{Remark}

\newtheorem{Non-example}{Non-example}

\date{\today}
\title{Systolic pants decompositions with polynomial growth}
\author{Tarik Aougab, Jonah Gaster, Nicholas Miller}
\date{}

\begin{document}

\begin{abstract} 
We investigate improved bounds on the kissing number of a hyperbolic surface under the assumption that the systole set contains pants decompositions with volume growth constraints.
Namely, when the dual graph of the pants decomposition has polynomial volume growth, we obtain a power-savings improvement for the upper bound on the kissing number of $S$ in terms of its systolic length. 
We also construct super-exponentially many (in $g$) pants decompositions with exponential volume growth 
for which our techniques still imply stronger upper bounds for kissing numbers.
The proofs use the technology of ``stem systems'' and ``tulips'' developed by the first and second author to study $1$-systems of curves.  

\end{abstract}

\maketitle

\section{Introduction}

A \textit{systole} of a Riemannian manifold $(M,g)$ is an essential closed geodesic of minimal length, and the \textit{systolic length} $\mathrm{sys}(M)$ is the length of any such geodesic.
When $M=S_g$ is a closed hyperbolic surface, the study of systoles is a classical problem in hyperbolic geometry, related more broadly to spectral geometry.
For a fixed genus $g$, Mumford's compactness criterion \cite{Mumford} shows that $\mathrm{sys}$ attains a maximum on $\mathcal{M}_g$, the moduli space of hyperbolic structures, and it is natural to ask for this maximal value as well as for the surfaces that attain it.
As with many extremal problems on moduli space, determining this value explicitly is extremely difficult; in genus $2$, work of Jenni and Schmutz Schaller (\cite{Jenni}, \cite{Schmutz}) shows that the maximum is uniquely attained by the Bolza surface, while the corresponding problem remains open in every higher genus.

Rather than studying these exact values, it is perhaps easier to consider the asymptotic problem of determining the maximal systolic length.
Indeed, a simple argument using the area of a hyperbolic disk gives an asymptotic upper bound with leading term $2\log g$.
On the other hand, Buser and Sarnak \cite{BuserSarnak} constructed arithmetic families arising from principal congruence subgroups for which
\[
\mathrm{sys}(S_g)\ge \frac{4}{3}\log g-C,
\]
for some constant $C>0$ independent of $g$.
Katz, Schaps, and Vishne \cite{KSV} extended this construction to principal congruence towers associated to arbitrary closed arithmetic hyperbolic surfaces, while Makisumi \cite{Makisumi} showed that the coefficient $4/3$ is optimal for principal congruence towers associated to a fixed closed arithmetic hyperbolic surface.
Thus, at present,
\[
\frac{4}{3}
\le
\limsup_{g\to\infty}
\frac{\max_{S\in\mathcal{M}_g}\mathrm{sys}(S)}{\log g}
\le
2.
\]
A complementary extremal problem asks not how long the shortest closed geodesic can be but how many such geodesics a surface can possess.

If $S\in\mathcal{M}_g$ is a closed hyperbolic surface, the \textit{kissing number} $\K(S)$ is the cardinality of the set of systoles on $S$.
As with the corresponding problem for the systolic length, it is of particular interest to compute the maximal value of $\K(S)$ as $S$ varies over $\mathcal{M}_g$.
Unfortunately, doing this explicitly is again an extremely difficult problem.
For instance, in genus $2$, the kissing number is uniquely maximized by the Bolza surface, which has $12$ systoles \cite{MRT}.
Conjecturally, in genus $3$, the kissing number is maximized by the $M(3)$ surface initially studied by Schmutz Schaller \cite{Schmutz}. 
However, this is the extent of present knowledge of exact maximal values of this quantity, and we are again naturally led to study the asymptotic problem of determining the precise relationship between the maximal value of $\K(S)$ and the genus $g$.

For a genus $g$ surface, the fact that the maximal kissing number is at least linear in $g$ is straightforward; one can choose a metric on $S_g$ with an extremely short pants decomposition consisting of curves of equal length.
The collar lemma then implies that no other curve will be shorter.
The first result improving on this simple lower bound was given by Schmutz Schaller \cite{Schmutz3}.
To state this, here and throughout the rest of the paper, by a family $\cS$ of hyperbolic surfaces, we mean a collection $(S_{g_i})$ of surfaces of genus $g_i$, where $g_i\to\infty$.
In this language, Schmutz Schaller proved that, for any $\epsilon>0$, there is a family $\cS$ of hyperbolic surfaces such that $g^{4/3-\epsilon}\le \K(S)$ for each $S\in\cS$ of genus $g$.
In particular, one finds that
\[
\limsup_{g\to\infty}
\frac{\log\left(\max_{S\in\mathcal{M}_g}\K(S)\right)}{\log g}
\ge \frac{4}{3}.
\]
As is perhaps natural, these examples come from taking congruence towers of certain arithmetic hyperbolic surfaces.

On the other hand, a classical upper bound is due to Parlier \cite{parlier}.
For $S\in\mathcal{M}_g$, let $\ell=\mathrm{sys}(S)$ denote its systolic length.
Parlier establishes that
\begin{equation}\label{eqn:Parlierbound}
\K(S) \leq C  g  \frac{e^{\ell/2}}{\ell},
\end{equation}
for some universal constant $C>0$.
In particular, Parlier deduces that there exists a universal $D>0$ such that
\[
\K(S) \leq D \frac{g^2}{\log g},
\]
for every $S\in\mathcal{M}_g$. Petri and Bourque use Selberg trace formula techniques to generalize this upper bound bound to higher dimensional hyperbolic manifolds \cite{bourque}. 

Conjecturally, this upper bound is not sharp.
As some evidence for this, note that the constructions above have $\ell\approx \frac43\log g$, which is expected to be asymptotically maximal, and therefore Parlier's bound takes the form $O\left(g^{5/3}/\log g\right)$.
In fact, Schmutz Schaller conjectures that even this is not sharp and that the correct exponent should be $4/3$ (see \cite{Schmutz3}).

The focus of this paper is to investigate improvements to the upper bound for $\K(S)$ in Equation \eqref{eqn:Parlierbound} under additional assumptions on the systole set of $S$.
In particular, we consider families of surfaces with a systolic pants decomposition whose dual graphs possess polynomial volume growth (see below for the precise definition).
The motivation comes from the expectation that kissing numbers should be smaller on such surfaces, since the examples with the largest known kissing numbers are constructed from congruence covers of certain arithmetic hyperbolic orbifolds.
At least heuristically, the combinatorics of a short pants decomposition on one of these congruence covers should reflect the geometry of a Cayley graph for $\SL_{2}(\mathbb{Z}/n\mathbb{Z})$.
The corresponding Cayley graphs form an expander family as $n\to\infty$ and therefore have exponential volume growth on the relevant scales, in contrast with the polynomial volume growth assumed here.
Our main result is that, when $S$ admits such a polynomially growing pants decomposition, one can obtain power-saving improvements on Parlier's upper bound for $\K(S)$ in Equation \eqref{eqn:Parlierbound}.

We now define this precisely.
Given a non-decreasing function $f:[0,\infty)\rightarrow[0,\infty)$, we say that a graph $G$ has \emph{volume growth $f$} if, for each $r>0$ and each vertex $v\in G$, we have $\vol(B_r(v,G))\le f(r)$.
A graph $G$ has \textit{polynomial volume growth of degree $d$} if there is a monomial $f$ of degree $d$ such that $G$ has volume growth $f$.
This notion can just as easily be applied to a family of graphs $(G_n)_{n=1}^{\infty}$, which is said to have volume growth $f$ precisely when each graph $G_n$ does.
Finally, we say that a family $(G_n)_{n=1}^{\infty}$ has \emph{polynomial volume growth of degree $d$ at threshold $R$} provided there exists a monomial $f$ of degree $d$ such that, for every $n$ and every vertex $v\in G_n$, we have $\vol(B_R(v,G_n))\le f(R)$. 

\begin{Theorem} \label{thm:main}
Let $\cS$ be a family of hyperbolic surfaces such that each $S\in\cS$, with systolic length $\ell$, which admits a systolic pants decomposition whose dual cubic graph has polynomial volume growth of degree $d\ge 1$ at threshold $(\ell e^{\ell/2})^{1/(d+1)}$.
Then, for each $S\in\cS$ of genus $g$,
\[ 
\K(S) \le C  g \left(\ell  e^{\ell/2}\right)^{1-\frac{1}{d+1}}. 
\]
\end{Theorem}

We remark that for certain values of $d$, the use of a threshold is critical in order for this bound to be non-trivial. 
For example, when $d<2$, if one removes the threshold and requires the stronger condition that the dual graph has subquadratic growth at all scales, then techniques of Parlier show that $\K(S)$ is necessarily $O(g)$ (see Section \ref{sub:thresh} for the argument).

The proof of Theorem \ref{thm:main} utilizes the technology of stem systems introduced by the first two authors to study $1$-systems of curves. In \cite{AougabGaster}, stem systems are used to produce a quadratic-in-$|\chi|$ upper bound on the maximum size of a collection of pairwise non-homotopic essential simple closed curves, such that any two intersect at most one time (i.e.~a \emph{$1$-system}). 
The connection to systole sets is that, when $S$ is closed, the systoles form a $1$-system. 
On the other hand, the proof of Theorem \ref{thm:main} also uses the systolic hypothesis in an essential way. In particular, this result is not simply a topological result about certain types of $1$-systems; it takes as additional input the geometric properties enjoyed by systoles, for example the systolic collar lemma, and the fact that systoles intersect in definitively large angles, as proved by Parlier \cite{parlier}. 

We note also that the assumption of admitting a systolic pants decomposition can be weakened; Theorem \ref{thm:main} applies as written whenever the surface admits a systolic multi-curve whose complementary components have uniformly bounded Euler characteristics. Going forward, we will call such a multi-curve \textit{large}. Even more generally, we can prove a general version of Theorem \ref{thm:main} which we now state.

\begin{Theorem}\label{thm:generalmain}
Let $\cS$ be a family of hyperbolic surfaces such that each $S\in\cS$, with systolic length
$\ell$, which admits a large systolic multi-curve $m$ such that the absolute value of the Euler characteristic of each complementary region is at most $T$ and whose dual graph has polynomial volume growth of degree $d$ at the 
threshold 
$\left(\ell e^{\ell/2}/T\right)^{1/(d+1)}$,  then
\begin{equation} \label{eq:gen thm}
 \K(S) \leq C  g T^{\frac1{d+1}} \left(\ell   e^{\ell/2}\right)^{1 - \frac{1}{d+1}}. 
 \end{equation}
 \end{Theorem}

\begin{remark}
When the systole set contains a pants decomposition of linear growth, our bound is on the order of $g e^{\ell/4}$. Letting $X_{n}$ denote the cover of the modular orbifold associated to the principal congruence subgroup $\Gamma(n) < \mathrm{PSL}_{2}(\mathbb{Z})$, the kissing number is larger than $|\chi(X_{n})|^{4/3-\epsilon}$ for all $\epsilon >0$ and the systolic length $\ell_{n}$ is approximately $(4/3) \log(|\chi(X_{n}|)$. Thus the kissing number is larger than $|\chi| e^{\ell_{n}/4 - \epsilon}$, nearly matching this upper bound. 

Of course, $X_{n}$ does not admit a systolic pants decomposition of linear volume growth, so 
these estimates are not actually comparable. 
That is, though the linear case of our theorem seems to match the observed maximal growth of $\K(S)$, it is not possible to have large $\K(S)$ in this setting.
\end{remark}

Obviously the assumption of polynomial growth at only a specific threshold is weaker than assuming polynomial growth but it is natural to wonder whether there are actual examples where Therorem \ref{thm:main} applies but polynomial growth of the corresponding dual graphs fails dramatically. In fact, one can build families of surfaces admitting systolic pants decompositions that have quadratic growth at the appropriate threshold, and therefore Theorem \ref{thm:main} applies in one of its strongest forms, but which have exponential volume growth. One can also tweak the examples so that the volume growth is linear at the appropriate threshold but then Lemma~\ref{lem:GPTlemma} will show that the kissing number is at most linear.

\begin{Theorem} \label{thm:construction} Given $\ell: \mathbb{N} \rightarrow \mathbb{R}_{+}$, there is $\epsilon>0$ and a super-exponentially large (in $g$) collection of cubic graphs $\cG_{g}$ so that 
\begin{enumerate}
\item The volume growth function for any $G \in \cG_{g}$ is exponential with growth rate $\Omega(\ell^{-2}e^{-\ell/2})$;

\item For any hyperbolic surface $S$ admitting a systolic pants decomposition $\cP$ with dual graph in $\cG_{g}$ and systolic length $\ell(g)$ one has 
\[ \K(S) \leq g  e^{\ell/2  \left(1 - \epsilon \right)}~.\]
\end{enumerate}
\end{Theorem}

Finally, for any $S$ admitting a systolic pants decomposition $\cP$ of any volume growth, we investigate the size of the set of systoles with small \textit{combinatorial length} -- the number of curves in $\cP$ a particular systole crosses. Of course, any such surface has the property that there are $3g-3$ systoles with combinatorial length $0$ -- these are the curves in $\cP$ themselves. We prove that these are essentially the only systoles with combinatorial length $\ll \log(g)$ in the following.

\begin{Theorem} \label{thm: small comb length} Let $f(g) = o(\log(g))$ and given $S$ with systole set $\Gamma$ containing a systolic pants decomposition $\cP$, let $\cC^{S, \cP}_{f}$ be 
\[
\cC^{S, \cP}_{f} = \left\{\, \gamma \in \Gamma \;|\; i (\gamma, \cP) \leq f(g) \,\right\}. 
\]
Then 
\[ | \cC^{S, \cP}_{f}| = o(g^{1+\epsilon}),\quad  \forall \epsilon >0.  \]

\end{Theorem}

\subsection{The need for stems.} In order to highlight the necessity of our techniques, we quickly summarize the logic of our argument and point out where it requires a genuinely new ingredient. Each systole on $S$ corresponds to a simple cycle on the dual graph $G_{\cP}$. We can then partition systoles into two families, $\fL$ and $\fS$: 
systoles in $\fL$ have large combinatorial length -- say at least $e^{\epsilon \ell/2}$, for some $\epsilon$ -- while the complementary systoles in $\fS$ have small combinatorial length. Note however that the correspondence from systoles to cycles on the dual graph is not necessarily one-to-one. Indeed, the fiber over a cycle of combinatorial length $L$ consists of a $1$-system of curves on a subsurface of $S$ homeomorphic to $S_{1,L}$, a torus with $L$ boundary components.
Such a 1-system
can have size 
$3L$ (in fact this is the maximum size of a 1-system on $S_{1,L}$ \cite[Thm.~1.2]{MRT}).  

Implicit in Parlier's work is the statement that any systole intersects at most $\ell  e^{\ell/2}$ many other systoles. Therefore, since each curve in $\mathfrak{L}$ has to contribute at least 
$e^{\epsilon \ell/2}$ intersections 
with $\cP$, its size is at most the number of curves in $\cP$ times $\ell e^{\ell/2}$, divided by $e^{\epsilon \ell/2}$, which is on the order of the desired power-savings bound.

As for $\mathfrak{S}$, one could hope to bound the number of such curves by thinking solely about the combinatorics of $G_{\cP}$ and using the polynomial growth assumption. However it is easy to construct examples of cubic graphs with the property that in a ball of radius $r$ about a vertex $v$, there are \textit{exponentially-in-}$r$ many simple cycles through $v$ even when the volume growth is only linear. Combining this difficulty with the fact that the correspondence between systoles and cycles on the dual graph is many-to-one, we see that there is no obvious path forward. 

Instead, we use the fact that any stem system determined by curves in $\mathfrak{S}$ must necessarily land in a subsurface of $S$ with Euler characteristic at most the volume of a ball in the dual graph of radius on the order of $e^{\epsilon\ell/2}$ (see Section \ref{Sec: prelim} for definitions). The techniques developed in \cite{AougabGaster} then imply the desired upper bound on $|\mathfrak{S}|$.

\subsection{Outline.} In Section \ref{Sec: prelim} we introduce necessary terminology and results from \cite{parlier} and \cite{AougabGaster}. In Section \ref{Sec: main thm}, we prove Theorem \ref{thm:main} and Theorem \ref{thm: small comb length}. Finally, in Section \ref{Sec: construction}, we describe the cubic graphs with exponential volume growth for which Theorem \ref{thm:construction} applies. 

\subsection{Acknowledgements.} The idea for the proof of Lemma \ref{lem:GPTlemma} was obtained by an LLM which will remain unnamed (but of course we are happy to reveal the one used upon request). The rest of the paper was completely human-generated and the entirety of the paper -- including the proof of Lemma \ref{lem:GPTlemma} -- was written by humans, namely the authors. 
Miller was partially supported by NSF grant DMS--2405264/2630020.

\section{Preliminaries} \label{Sec: prelim}

\subsection{Parlier's bounds}
In \cite{parlier}, Parlier gave a significant improvement over prior kissing number bounds by showing that there is an upper bound on the order of $O(g^2/\log(g))$ for compact Riemann surfaces.
This was generalized to finite-type surfaces of any signature by Fanoni--Parlier \cite{FanoniParlier}.
We now recount Parlier's argument as it will motivate some of our techniques as well as the scales we work on.

In \cite[Theorem 1.3]{parlier}, Parlier shows a relation between the systole of a surface and the kissing number. 
More explicitly, if $S_g$ is a hyperbolic surface of genus $g$ with systole of length $\ell$, then Parlier proves that 
\begin{equation}\label{eqn:Parlier}
\K(S)\le Cg\frac{e^{\ell/2}}{\ell},
\end{equation}
for some universal constant $C$.
An asymptotic of the form $O(g^2/\log(g))$ follows immediately from the well-known bound $\ell\le 2\log(4g-2)$ for the length of the systole.
\begin{remark}
The observant reader will rightfully be concerned about the factor of $\ell$ in the denominator of this upper bound; however using only the collar lemma one can prove the kissing number is at most $3g-3$ when $\ell\le 2\sinh^{-1}(1)$.
\end{remark}
The upper bound in Equation \eqref{eqn:Parlier} is obtained using an analysis of the geometry of systoles and a packing argument.
First, \cite[Lemma 2.2]{parlier} gives an improved collar lemma for systoles which shows that disjoint systoles are separated by distance at least $\Omega(e^{-\ell/4})$.
Second, \cite[Lemma 2.4]{parlier} shows that if two systoles intersect then they must do so with angle at least $\Omega (e^{-\ell/4})$.
A standard packing argument, together with Gauss--Bonnet, shows that $S$ can be covered by $O(ge^{\ell/2})$ balls of radius comparable to $e^{-\ell/4}$.
The reason this scale of radius is chosen is because if a collection of systoles meets one such ball then they must pairwise intersect by the disjointness lemma mentioned above.
Using the lower bound on the angle of intersection together with an additional elementary estimate, Parlier shows that at most $O(e^{\ell/4})$ systoles can meet one such ball, whereas each systole meets at least $\Omega(\ell e^{\ell/4})$ balls.
Consequently an upper bound on the kissing number is given by
$$\K(S)\le C\frac{(ge^{\ell/2})(e^{\ell/4})}{\ell e^{\ell/4}}=Cg\frac{e^{\ell/2}}{\ell},$$
which yields Equation \eqref{eqn:Parlier}

\subsection{Stem systems}
We now recall the construction of a stem system from a 1-system $\Gamma$ of geodesics on a hyperbolic surface $S$ containing a pants decomposition $\cP$ \cite{AougabGaster}. Note that the preimage of $\cP$ in the universal cover $\bH$ is a collection of disjoint complete geodesics and we denote the dual tree by $\cG_{\cP}\subset \bH$.
Each geodesic $\gamma\in \Gamma- \cP$ lifts to a geodesic in $\bH$, which, when oriented, determines a trajectory of edges in $\cG_{\cP}$.

For each pants curve $\alpha\in\cP$, let $\Gamma_\alpha=\{\gamma\in\Gamma:i(\gamma,\alpha)>0\}$. 
The annular cover $A_\alpha\to S$ has core $\tilde{\alpha}\to \alpha$, and, because $\Gamma$ is a 1-sytem each of the curves $\gamma\in\Gamma_\alpha$ determines a unique lift $\tilde{\gamma}:\bR\to A_\alpha$ that intersects $\tilde{\alpha}$.

Assume for the moment that each of the pants curves is oriented.
The collection of lifts $\{\tilde{\gamma}:\gamma\in\Gamma_\alpha\}$ determine points on the component of $\partial A_\alpha$ lying to the left of $\alpha$, a copy of $S^1$. We obtain an induced cyclic order on $\Gamma_\alpha$.
In this way, each pants curve $\alpha$ determines two cyclic orders on $\Gamma_\alpha$; observe that these two cyclic orders may be quite different from each other, as the position of curves in $\Gamma_\alpha$ with respect to this cyclic order may change places on either side if they form a triangle with $\alpha$.

Now choose $\gamma\in \Gamma- \cP$, and endow $\gamma$ with an orientation. Because $\cP$ is a pants decomposition, there is some $\alpha\in\cP$ so that $\gamma\in\Gamma_\alpha$, and suppose that $\alpha$ is oriented so that $\gamma$ and $\alpha$ intersect positively.
Suppose that $\eta\in \Gamma_\alpha$ is the immediate successor of $\gamma$ in the cyclic order on $\Gamma_\alpha$ obtained with the chosen orientation of $\alpha$. 
Choose a lift $\tilde{\alpha}$ of $\alpha$ in $\bH$, and consecutive incident lifts $\tilde{\gamma}$ and $\tilde{\eta}$ of $\gamma$ and $\eta$, respectively. 
By construction, the edge trajectories of $\tilde{\gamma}$ and $\tilde{\eta}$ in $\cG_{\cP}$ intersect. 
Therefore there is a maximal shared subsequence $\tilde{c}_0,\ldots,\tilde{c}_n$ of edges in $\cG_{\cP}$ for $\tilde{\gamma}$ and $\tilde{\eta}$; we indicate the distinct edges in $\cG_{\cP}$ of the trajectories of $\tilde{\gamma}$ and $\tilde{\eta}$ that immediately follow $\tilde{\gamma}$ and $\tilde{\eta}$ as $\tilde{c}^\gamma$ and $\tilde{c}^\eta$, respectively.

The \emph{tulip} with main sides $\gamma$ and $\eta$ is given by $Q\to S$, where $Q\subset \bH$ is the convex hull of the four points $\tilde{\gamma}\cap \tilde{c}_0$, $\tilde{\eta}\cap \tilde{c}_0$, $\tilde{\gamma}\cap \tilde{c}^\gamma$, and $\tilde{\eta}\cap \tilde{c}^\eta$.
We call $\tilde{c}_0$ (or its image in $S$) the \emph{base} of the tulip.
For any point $p\in Q$, there is a geodesic from $p$ to the midpoint of the base $\tilde{c}_0$ (again, or its image in $S$); this is called a \emph{stem} at $p$.

It turns out that stems at $p$ are most convenient when $p$ is chosen from a certain subset of $Q$.
The tulip $Q$ contains a \emph{flower} $N\subset Q$, formed by the convex hull of the four points $\tilde{\gamma}\cap \tilde{c}_n$, $\tilde{\eta}\cap \tilde{c}_n$, $\tilde{\gamma}\cap \tilde{c}^\gamma$, and $\tilde{\eta}\cap \tilde{c}^\eta$. 
Finally, the flower $N$ contains an \emph{inner flower} $N^\Delta\subset N\subset Q$, formed by the complement in $N$ of the three convex hulls 
\[
\Hull(\tilde{c}_0\cup\tilde{c}^\gamma) \cup \Hull(\tilde{c}_0\cup\tilde{c}^\eta) \cup \Hull(\tilde{c}^\gamma\cup\tilde{c}^\eta)~.
\]
Observe that the inner flower could in fact be empty.
However, as the hyperbolic lengths of the curves in $\cP$ shrink, the inner flower $N^\Delta$ limits geometrically to an ideal triangle. That is, ``often, most of $Q$" is the inner flower.

Observe that a stem at $p$ of a tulip is \emph{almost} homotopic to a portion of the main sides of the tulip, which are simple closed curves on $S$ -- this is of course imprecise because of the problematic point $p$, which either doesn't move during a homotopy, in which case the stem cannot be homotopic to a main side, or does move, in which case the stem is contractible. 
Nonetheless, it should not be surprising that stems are usually simple as well:

\begin{lemma}[Prop.~2.14,\cite{AougabGaster}]
\label{lem:stems embed}
Stems from points in the the inner flower of a tulip embed in $S$.   
\end{lemma}

In addition to being embedded arcs from $p$ to $\cP$, it turns out that tulips from 1-systems on $S$ can be massaged so that the associated stem systems have some strong restrictions. 
Given a collection $\cN$ of tulips on $S$, a tulip $Q\to S$ from $\cN$ is \emph{maximal} provided no stem at $p$ from the family $\cN$ is a prefix of the $Q$-stem at $p$. Here ``prefix" must be interpreted up to homotopy in the following precise sense: For stems $\sigma_1,\sigma_2:(I,\partial I)\to (S,\{p\}\cup\cP)$, $\sigma_1$ is a prefix of $\sigma_2$ provided $\sigma_1$ is homotopic, as a map of pairs, to the restriction $\sigma_2|_{(J,\partial J)}$, for a subinterval $J\subset I$. 
A $(p,\cP)$-\emph{stem system} $\cS$ on $S$ is the set of stems at $p$ associated to a collection of maximal tulips based at $\cP$.

In \cite[Lem.~5.1]{AougabGaster} it is shown that it is possible to select a subset $\cN'$ from the collection $\cN$ of tulips associated to a 1-system which decomposes into a finite union $\cN'=\cN_1\sqcup\ldots \sqcup \cN_{21}$, with the property that each $\cN_i$ is a collection of maximal tulips, whose corresponding stems then form a stem system. 
This is done by first 
partitioning $\cN$ according to the \emph{refined type} of a tulip, which refers to the amount of wrapping the flower does around its base, and selecting only the tulips which are \emph{type-maximal}; see \cite[\S3]{AougabGaster} for details. 

In particular, for each $\cN_i$ we find a stem system $\cS_i$, where:
\begin{enumerate}
    \item Each stem in $\cS_i$ is an embedded arc from $p$ to $\cP$ that intersects each cuff at most once. 
    \item No two stems in $\cS_i$ are homotopic.
    \item No stem in $\cS_i$ is homotopic to a prefix of another stem in $\cS_i$.
    \item If $\sigma_1,\sigma_2\in \cS_i$ intersect, then they form a triangle with the base of one of the two stems.
\end{enumerate}

Finally, we will require the following slight generalization of the counting strategy in \cite{AougabGaster}:

\begin{Prop}
\label{prop:stem systems bound}
Suppose that $\cS$ is a $(p,\cP)$-stem system, where $p$ lies in the vertex $v$ of the dual graph $G_\cP$, and so that each of the stems intersects at most $R$ pants curves. 
Then there is an absolute constant $C>0$ so that 
\[|\cS|\le C \,\vol(B_R(v,G_\cP)).\]
\end{Prop}

\begin{proof} In \cite[Prop.~4.2]{AougabGaster}, the first two authors establish a bound on the size of a stem system that is linear in the Euler characteristic of the entire surface. The proof begins by considering a lift $\tilde{p}$ of $p$ to the universal cover and corresponding lifts of each stem. Because no stem is a prefix of any other, the lifts of curves in $\mathcal{P}$ containing the endpoints of the stems cannot nest from the perspective of $\tilde{p}$. For this reason, these lifts are naturally cyclically ordered around $\tilde{p}$ and we can connect successive endpoints of stems by geodesics to form a convex polygon. The proof works by showing that the lion's share of the area of that polygon embeds under the covering projection; since its area is a function of the number of stems, we are done. 

To obtain Proposition \ref{prop:stem systems bound}, one only needs to observe that if each stem intersects at most $R$ curves in $\mathcal{P}$, the entire stem system is contained in a subsurface $Y$ around $p$ consisting of those points which can be reached from $p$ by a path that crosses at most $R$ pants curves. The Euler characteristic of $Y$ is proportional to $\vol \; B_R(v, G_\mathcal{P})$ and thus so is its area. One then simply carries out the argument summarized in the first paragraph but restricted to $Y$: the area of the polygon in the universal cover is proportional to the number of stems, but now a majority of this area will embed into $Y$. \qedhere

\end{proof}

Finally, we can observe that Proposition \ref{prop:stem systems bound} and its proof does not at all rely on the assumption that $\mathcal{P}$ is a pants decomposition. We can instead replace it with a multi-curve at the cost of losing some of the control between the volume of a ball in the dual graph and the Euler characteristic of the corresponding subsurface (since now, a single vertex of the dual graph might correspond to a subsurface that is much larger than a single pair of pants). In particular, one gets a generalization of Proposition \ref{prop:stem systems bound} in which the volume of a ball in the dual graph is replaced by the size of the subsurface it represents: 

\begin{Prop}
\label{prop:better bound}
Suppose that $\cC$ is a multicurve and $\cS$ is a $(p,\cC)$-stem system, where $p$ lies in the vertex $v$ of the dual graph $G_\cC$, and so that each of the stems intersects at most $R$ curves in $\cC$.
Then there is an absolute constant $C>0$ so that
\[
|\cS|\le C\,  \xi(B_R(v,G_\cC))~,
\]
where $\xi (H)$ indicates the absolute value of the Euler characteristic of the subsurface of $S$ determined by the subgraph $H\subseteq G_\cC$.
\end{Prop}

\section{Proof of Theorems \ref{thm:main}, \ref{thm:generalmain}, and \ref{thm: small comb length}} \label{Sec: main thm}

In this section, we'll prove Theorem \ref{thm:main}.
For this, let $S$ be a closed hyperbolic surface of genus $g$ admitting a systolic pants decomposition $\cP$, let $G_\cP$ be its dual graph, let $\Gamma$ be its set of systoles, and let $\ell$ be its systolic length.
We may assume throughout that $\ell\ge 2\operatorname{arcsinh}(1)$, since otherwise 
the Collar Lemma implies that all systoles are disjoint and hence $|\Gamma|$ is $O(g)$.

As outlined in the introduction, we partition $\Gamma$ into two subsets $\mathfrak{L}$ and $\mathfrak{S}$ based on combinatorial length.
To this end, let $R>0$ be a placeholder for the threshold.
Define $\mathfrak{L} \subset \Gamma$ as the subset of all systoles with combinatorial length at least $R$ and define $\mathfrak{S}$ as its complement. 
We show that $|\mathfrak{L}|$ is $O(g\ell \frac{e^{\ell/2}}{R})$ and $|\mathfrak{S}|$ is $O(gV_R)$, where $V_R=\max_{v\in G_\cP} \vol(B_R(v,G_\cP))$.
Consequently, we will see that
$$|\Gamma|\le C g\left(\ell \frac{e^{\ell/2}}{R}+V_R\right).$$
In particular, when 
$R=(\ell e^{\ell/2})^{1/(d+1)}$ and $G_\cP$ has polynomial growth of degree $d$ at threshold $R$, then $V_R\le AR^d$ for some universal $A>0$ and we conclude that
\begin{align*}
|\Gamma|&\le C g\left(\ell \frac{e^{\ell/2}}{R}+V_R\right) \\
&\le D g \left(\ell e^{\ell/2}\right)^{d/(d+1)}.
\end{align*}
This will complete the proof of Theorem \ref{thm:main} upon proving these two estimates.

We first bound $|\mathfrak{L}|$.
Cover each $\alpha \in \cP$ by $O(\ell  e^{\ell/4})$ balls of radius $\Theta(e^{-\ell/4})$. Using Parlier's systolic collar lemma, in particular the proof of \cite[Theorem 2.9]{parlier}, one sees that $O(e^{\ell/4})$ elements of $\Gamma$ can cross $\alpha$ in a given ball. Therefore, 
\[ i(\alpha, \Gamma) = O(\ell  e^{\ell/2}),\]
for all $\alpha \in \cP$.
Since a pants decomposition has $\Theta(g)$ curves, we conclude that
\[ i(\cP, \Gamma) = O(g  \ell  e^{\ell/2}).\]
By definition of combinatorial length, every element of $ \mathfrak{L}$ contributes at least $R$ to this sum and therefore
$$R|\mathfrak{L}|\le C_1g\ell e^{\ell/2}.$$
This shows that
$$|\mathfrak{L}|\le C_1g\ell \frac{e^{\ell/2}}{R},$$
completing the first estimate.

Finally, we prove the bound on $|\mathfrak{S}|$.
We stress that, 
for the remainder of the argument, we will not use any geometric properties of the geodesics in $\mathfrak{S}$; we will only require the topological facts that $\mathfrak{S}$ is a $1$-system, that $\cP\subset\mathfrak{S}$ is a maximal multicurve, and that each curve in $\mathfrak{S}$ intersects $\cP$ fewer than $R$ times.
For this reason, we are fully permitted to change the metric on $S$ to a new one with respect to which the curves in $\mathfrak{S}$ need not be systoles. 

In this vein, we pick a metric on $S$ for which the curves in $\cP$ are extremely short and replace all curves in $\mathfrak{S}$ by their geodesic representatives.
We then consider the set of equivalence classes of tulips determined by $\cP$ and the elements of $\mathfrak{S}$.
Each such tulip corresponds to a pair of elements in $\mathfrak{S}$ that are adjacent in the circular order at infinity in an annular cover corresponding to an element of $\cP$.
As shown in \cite{AougabGaster}, each element of $\mathfrak{S}\setminus\cP$ is a main side of a tulip that is \textit{maximal amongst its type}.
Moreover, for each fixed type, the set of such tulips can be bounded above by considering the map
\[
\bigsqcup_{N\in\mathcal{N}}N^\Delta
\longrightarrow
S,
\]
from the disjoint union of their inner flowers $N^\Delta$ to the surface.

Indeed, \cite{AougabGaster} show that this immersion is at most $C_2\sigma(\mathfrak{S})$-to-$1$, for some universal constant $C_2$, where $\sigma(\mathfrak{S})$ is an upper bound for the size of any stem system associated to this collection of tulips.
As the lengths of the curves in $\cP$ tend to zero, the area of the inner flower of each tulip tends to $\pi$.
An upper bound on the number of tulips therefore follows from the Gauss--Bonnet theorem.
Summing over the finitely many possible types gives
\[
|\mathfrak{S}|
\le
|\cP|
+
C_3|\chi(S)|\sigma(\mathfrak{S})
\]
for some universal constant $C_3$.

Thus, the desired bound on $|\mathfrak{S}|$ reduces to finding a similar bound on $\sigma(\mathfrak{S})$.
Since the curves in $\mathfrak{S}$ have combinatorial length less than $R$ and since the combinatorial length of a stem is always at most that of either main side of the corresponding tulip, each stem in a stem system has combinatorial length at most $R$ as well. 

By Proposition~\ref{prop:stem systems bound}, there is another universal constant $C_{4}$ so that 
\[
\sigma(\mathfrak{S}) \leq C_{4} \displaystyle \max_{v\in G_\cP}\vol (B_R(v,G_\cP)).
\]
By assumption, $G_\cP$ has polynomial volume growth of degree $d$, so finally
\[ |\mathfrak{S}| \leq C_5gV_R=C_5g\displaystyle \max_{v\in G_\cP}\vol (B_R(v,G_\cP)).  \]
This completes the proof of the second estimate.

Note that the same argument used to bound $|\mathfrak{S}|$ readily implies Theorem \ref{thm: small comb length}. Indeed, if every curve in this combinatorial collection has combinatorial length $o(\log(g))$, then since $G_\cP$ is cubic
$$V_R=\displaystyle \max_{v\in G_\cP}\vol (B_R(v,G_\cP))\le g^{o(1)}.$$
It therefore follows that the size of the corresponding collection is $o(g^{1+\epsilon})$ for every $\epsilon>0$.

It remains to address Theorem \ref{thm:generalmain} and prove the bound in Equation \ref{eq:gen thm}.
Suppose that $\cP$ is a maximal systolic multicurve and that every component of $S\setminus\cP$ has Euler characteristic of absolute value at most $T$.
The argument bounding $\mathfrak{L}$ is entirely unchanged, since it uses only that $\cP$ consists of $O(g)$ systoles and Parlier's local estimate.
To bound $\mathfrak{S}$, simply note that each vertex of the dual graph now corresponds to a subsurface whose Euler characteristic has absolute value at most $T$ (as opposed to $1$ for a pair of pants).
Therefore, by Proposition \ref{prop:better bound},
\[
\sigma(\mathfrak{S})
\le
C_4~ T\max_{v\in V(G_\cP)}\vol( B_R(v,G_\cP)),
\]
and consequently
\[ |\mathfrak{S}| \leq C_5gTV_R.  \]
In particular, we see that 
\[
|\Gamma|
\le
C g\left(
\ell\frac{e^{\ell/2}}{R}
+
TV_R
\right),
\]
and the desired bound now follows by taking 
$R=(\ell e^{\ell/2}/T)^{1/d+1}$ as before.

\section{Examples and Theorem \ref{thm:construction}} \label{Sec: construction}

The idea behind Theorem \ref{thm:construction} is simple and relies on the basic observation that the stem system-based bounds used in the proof of Theorem \ref{thm:main} did not require that the dual graph to $\cP$ had polynomial volume growth \textit{globally}. Indeed, the systolic collar lemma of \cite{parlier} implies that each curve in $\Gamma$ intersects $\cP$ at most $O(\ell  e^{\ell/4})$ times. Thus, the essence of Theorem \ref{thm:construction} is to build cubic graphs that have \textit{local} polynomial growth in a radius $\ell  e^{\ell/4}$ neighborhood of any vertex, but that nevertheless have exponential growth globally.

As a schematic, one should keep in mind a tree with many bivalent vertices, created by starting with a trivalent tree and then replacing each edge with a subgraph of polynomial growth. The resulting graph will of course not be $3$-regular, but as we outline below, it is easy to build $3$-regular graphs that mimic the structure of such a tree. 

Concretely, start with $\approx \log((2g-2)/(\ell  e^{\ell/4})^{2})$ layers of a trivalent tree $T$. Replace each end with a triangle as in the following figure: 

\begin{figure}[h]
\begin{center}
\includegraphics[height=4cm]{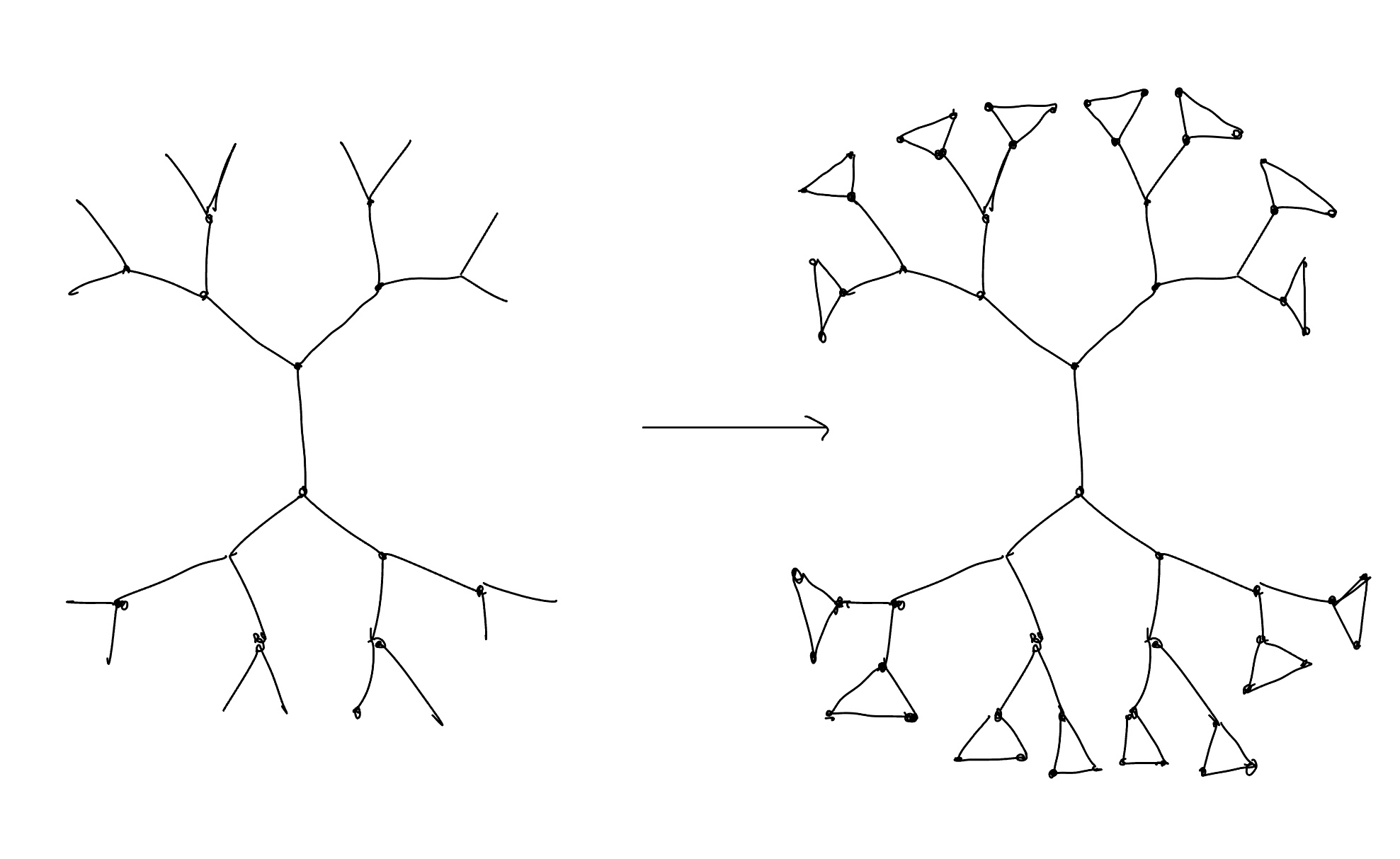}
\end{center}
\end{figure}

Next, choose a matching of the ends and glue together the bases of the  corresponding triangles.
We will refer to the edges of $T$ as the `internal edges' of the resulting graph.

\begin{figure} [h]
\begin{center}
\includegraphics[height=4cm]{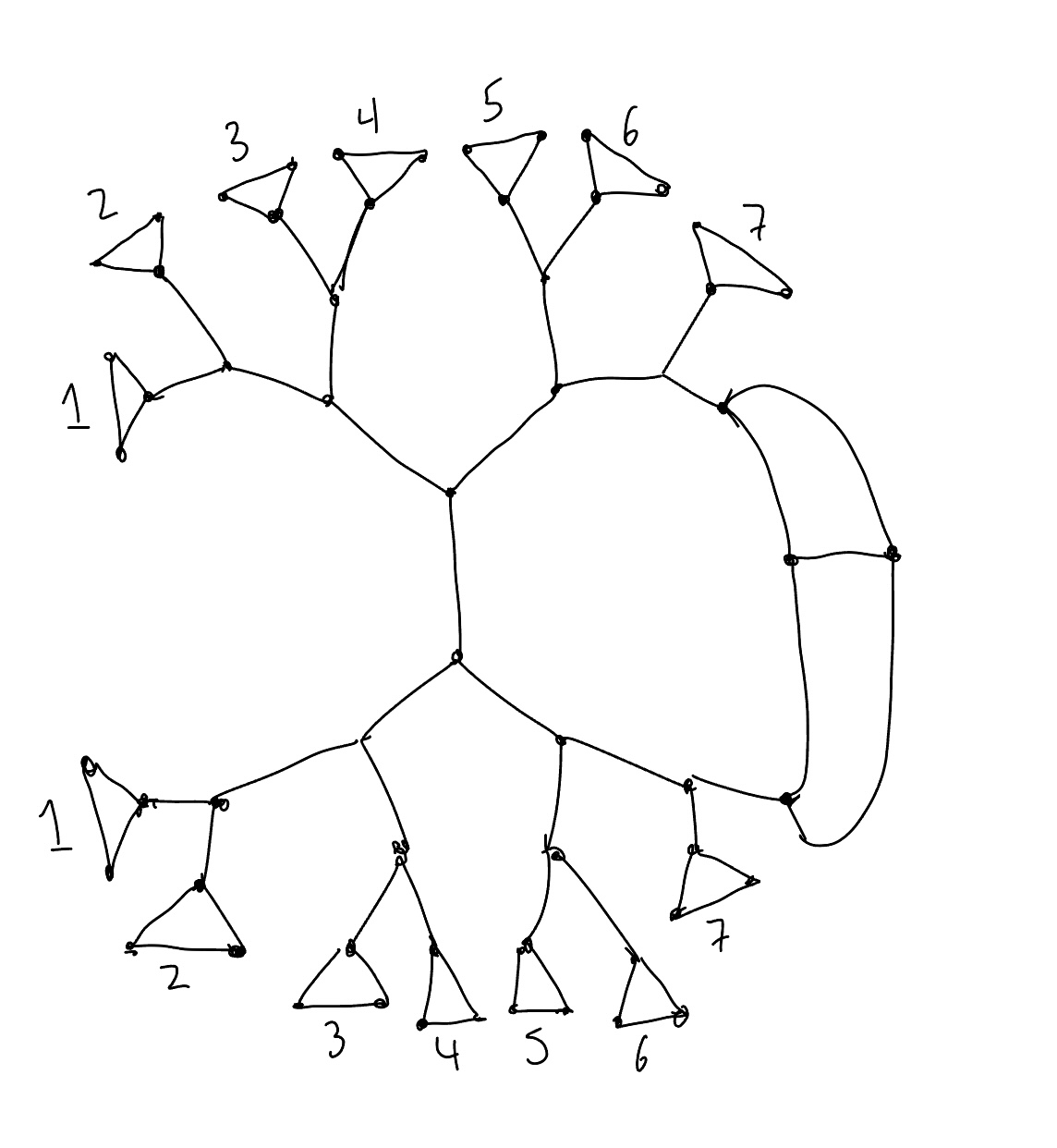}
\end{center}
\caption{Choose a perfect matching of the bases of each end-triangle as shown above. Glue corresponding bases together.}
\end{figure}

\vspace{50 mm}

We explain a last cut and paste operation. For each internal edge $e$ of $T$, we will replace $e$ with a graph $\mathcal{E}$,  of $\approx (\ell e^{\ell/4})^{2}$ edges, with two chosen leaves $\mathcal{E}^{\pm}$ in correspondence with the two endpoints of $e$. 
Our construction has the property that the endpoints $\mathcal{E}^\pm$
have metric neighborhoods of quadratic growth in $\mathcal{E}$. For instance, we can choose $\mathcal{E}$ to be a sequential union of ``ladders'' with increasing heights as shown in Figure \ref{fig:ladders}. The picture should be interpreted as being symmetric from the point of view of $\mathcal{E}^\pm$ (the two black vertices pictured near the bottom and the top right). 
So in particular, the heights of the ladders reach a maximum halfway through and then decline again towards the top right. 

\begin{figure} [h]
\begin{center}
\includegraphics[height=4cm]{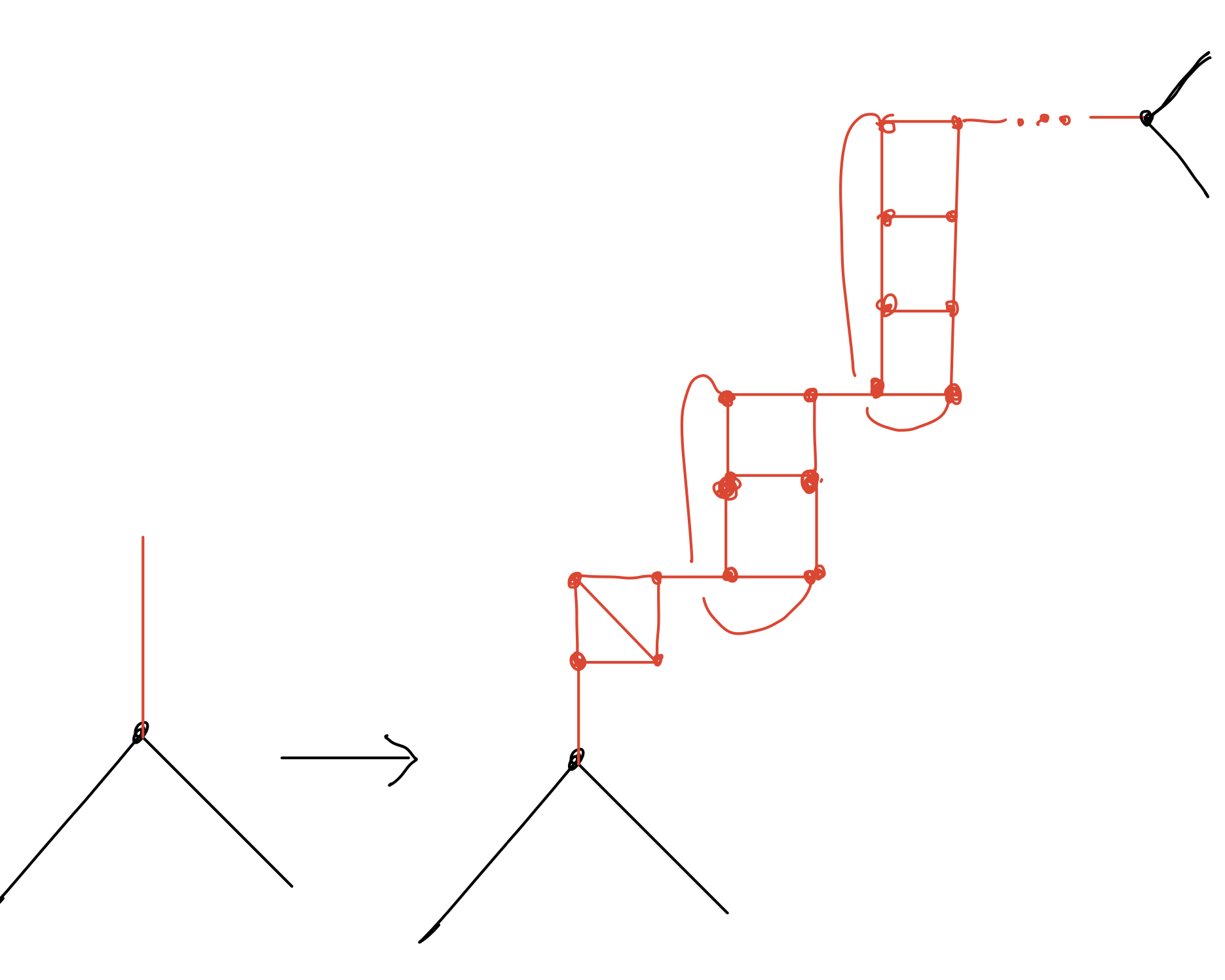}
\end{center}
\caption{}
\label{fig:ladders}
\end{figure}

Viewing the black vertex $v$ as 
one of the two leaves $\mathcal{E}^\pm$, note that every vertex in the union of the first $T$ ladders is accessible from $v$ by a path of length at most $3T/2$, and therefore $\mathcal{E}$ has quadratic growth. 

The resulting graph $\mathcal{G}$ is cubic and has $2g-2$ vertices. By construction, a ball $B$ of radius of $\ell  e^{\ell/4}$ in $\mathcal{G}$ has volume at most on the order of $(\ell  e^{\ell/4})^{2}$. Indeed, if $B$ is centered at an original and non-univalent vertex $v$ of $T$, $v$ must be the only original vertex within $B$. Since $\mathcal{E}$ has quadratic volume growth, the bound follows. 

If $B$ is centered at an end of $T$, a similar argument holds. Finally, if $B$ is centered at a new vertex in the middle of a ladder, the furthest it could possibly reach is into the middle of the $4$ adjacent copies of $\mathcal{E}$, and so at worst the leading coefficient of the volume polynomial increases. 

By choosing different matchings of the ends and by permuting the heights of ladders within each copy of $\mathcal{E}$, one obtains super-exponentially many isomorphism classes of cubic graphs that have quadratic growth up to the appropriate threshold. Finally, we note that by replacing $\mathcal{E}$ with a sequence of ladders whose heights grow quadratically (or more generally, polynomially of degree $d$) instead of linearly, we get cubic graphs that have \textit{cubic} (or of degree $d+1$) volume growth at the appropriate threshold, but which have exponential volume growth globally. 

\section{The necessity of a threshold}\label{sub:thresh}

As mentioned in the introduction, the use of a threshold is critical for certain non-trivial applications of Theorem \ref{thm:main}.
That is to say, for some values of $d$, removing the condition of a threshold forces all such examples to be covered by techniques from \cite{parlier}.

To make this precise, we say that a family $(G_n)_{n=1}^{\infty}$ has \emph{polynomial volume growth of degree $d$ at all scales} provided there exists a monomial $f$ of degree $d$ such that, for every $n$ and every vertex $v\in G_n$, we have $\vol(B_R(v,G_n))\le f(R)$ for every $R>0$.
In particular, such a family has subquadratic growth at all scales whenever $d<2$. 
We now show the following.

\begin{lemma}\label{lem:GPTlemma}
Let $\cS$ be a family of hyperbolic surfaces such that each $S\in\cS$, with systolic length $\ell$, admits a systolic pants decomposition whose dual cubic graph has subquadratic volume growth at all scales, then $\K(S)$ is $O(g)$.
\end{lemma}

\begin{proof}
For a given surface $S\in \cS$, consider the systolic pants decomposition $\cP$.
Each pair of pants has three cuffs of length $\ell$ and all pairs of pants have seams of the same length $s(\ell)$.
If $x,y\in S$ and $v_x$, $v_y$ are the two vertices in the dual graph $G_\cP$ corresponding to the pairs of pants that $x$ and $y$ lie in (if they lie on a cuff, then choose one of the two vertices arbitrarily), then one has the straightforward inequality that
$$(d_{G_\cP}(v_x,v_y)-1)s(\ell)\le d_S(x,y).$$
In particular, if  $d_S(x,y)\le \ell/2$
then 
$$d_{G_\cP}(v_x,v_y)\le 1+\frac{d_S(x,y)}{s(\ell)}\le 1+\frac{\ell}{2s(\ell)},$$
and consequently for any $x\in S$, if $y$ is in the ball of radius $\ell/2$ centered at $x$ on $S$, then $v_y\in B_{R(\ell)}(v_x,G_\cP)$ where $R(\ell)=1+\frac{\ell}{2s(\ell)}$.
Moreover, as $\ell$ is the systolic length, the ball of radius $\ell/2$ about any point is embedded in $S$.
Therefore by the subquadratic growth condition at all scales and the fact that the volume of a pair of pants is $2\pi$, we conclude that
$$\vol_S(B_{\ell/2}(x))= 2\pi(\cosh(\ell/2)-1)\le  2\pi \vol(B_{R(\ell)}(v_x,G_\cP)) \ll   2\pi A R(\ell)^2 ,$$
where the first term is the volume of the $\ell/2$ ball around $x$ on $S$.
Asymptotically, $\cosh(\ell/2)$ is on the order of $e^{\ell/2}$ and, by the usual hyperbolic hexagon identity, $R(\ell)$ is on the order of $\ell e^{\ell/4}$. In particular, $\ell$ must be universally bounded by some $L_0$ for all $S$.
However, then Parlier's bound in Equation \eqref{eqn:Parlier} immediately implies that $\K(S)$ is $O(g)$, provided that $\ell> 2\sinh^{-1}(1)$.
When this does not hold, systoles are pairwise disjoint and hence $\K(S)\le 3g-3$.
\end{proof}

\begin{remark}
A careful reading of the proof of Lemma \ref{lem:GPTlemma} shows that the same proof will work if one only assumes subquadratic growth at the threshold of $R(\ell)=\ell e^{\ell/4}$.
Therefore it is worth mentioning that for all $d > 1$, the threshold used in the statement of Theorem \ref{thm:main} is strictly asymptotically smaller than $R(\ell)$. In particular, for each fixed $d>1$, the growth hypothesis of Theorem \ref{thm:main} is strictly weaker than what is assumed in Lemma \ref{lem:GPTlemma}. 
\end{remark}

\bibliographystyle{abbrv}
\bibliography{kissing}
\end{document}